\documentclass[a4paper,onecolumn,10pt,accepted=2019-11-26, volume=1, issue=4, shorttitle=papers/compositionality-1-4]{compositionalityarticle}
\pdfoutput=1
\usepackage[T1]{fontenc}
\usepackage[utf8]{inputenc}
\usepackage{hyperref}
\usepackage{amsmath,amsthm}
\usepackage{amsfonts,amssymb,mathrsfs}
\usepackage[draft]{fixme}
\usepackage[numbers,sort&compress]{natbib}

\usepackage{graphicx}
\DeclareGraphicsExtensions{.pdf,.jpeg,.png}
\usepackage{csquotes}

\usepackage[all,2cell,cmtip]{xy}
\UseTwocells
\usepackage{tikz, tikz-cd}
\usetikzlibrary{backgrounds,circuits,circuits.ee.IEC,shapes,fit,matrix}

\tikzstyle{tri}=[regular polygon,regular polygon sides=3,shape border rotate=1
80,fill=none,draw=black]

\tikzstyle{simple}=[-,line width=2.000]
\tikzstyle{arrow}=[-,postaction={decorate},decoration={markings,mark=at position .5 with {\arrow{>}}},line width=1.100]
\pgfdeclarelayer{edgelayer}
\pgfdeclarelayer{nodelayer}
\pgfsetlayers{edgelayer,nodelayer,main}

\tikzstyle{none}=[inner sep=-1pt]

\definecolor{lblue}{rgb}{0,250,255}
\definecolor{llblue}{HTML}{a1ddde}
\definecolor{red}{rgb}{0.8,0,0}

\tikzstyle{species}=[circle, fill=yellow, draw=black, scale=1]
\tikzstyle{catalyst}=[circle, fill=yellow, draw=red, scale=1]
\tikzstyle{transition}=[rectangle, fill=lblue, draw=black, scale=1]
\tikzstyle{morphism}=[rectangle, fill=llblue, draw=black, scale=1]

\tikzstyle{empty}=[circle,fill=none, draw=none]
\tikzstyle{inputdot}=[circle,fill=black,draw=black, scale=.5]
\tikzstyle{dot}=[circle,fill=black,draw=black]
\tikzstyle{bounding}=[circle,dashed, fill=none,draw=black, scale=9.00]

\tikzstyle{simple}=[-,draw=black,line width=1.000]
\tikzstyle{inarrow}=[-,draw=black,postaction={decorate},decoration={markings,mark=at position .5 with {\arrow{>}}},line width=1.000]
\tikzstyle{tick}=[-,draw=black,postaction={decorate},decoration={markings,mark=at position .5 with {\draw (0,-0.1) -- (0,0.1);}},line width=1.000]
\tikzstyle{inputarrow}=[->,draw=black, shorten >=.05cm]

\newcommand{\maps}{\colon}

\newcommand{\N}{\mathbb{N}}

\newcommand{\define}[1]{{\bf \boldmath{#1}}}

\newcommand{\namedcat}[1]{\mathsf{#1}}
\renewcommand{\S}{\mathsf{S}}
\newcommand{\Petri}{\namedcat{Petri}}
\newcommand{\Cat}{\namedcat{Cat}}
\newcommand{\CMC}{\namedcat{CMC}}

\newcommand{\CMon}{\namedcat{CMon}}

\newcommand{\PreMonCat}{\mathsf{PreMonCat}}

\newcommand{\Set}{\namedcat{Set}}

\newcommand{\C}{\namedcat{C}}
\newcommand{\D}{\namedcat{D}}

\newcommand{\ob}{\mathrm{Ob}}
\newcommand{\mor}{\mathrm{Mor}}

\newcommand{\beq}{\begin{equation}}
\newcommand{\eeq}{\end{equation}}
\newcommand{\beqa}{\begin{eqnarray}}
\newcommand{\eeqa}{\end{eqnarray}}

\theoremstyle{plain}

\newtheorem{thm}{Theorem}

\newtheorem{prop}[thm]{Proposition}

\newtheorem{defn}[thm]{Definition}

\newtheorem{example}[thm]{Example}

\theoremstyle{definition}

\theoremstyle{remark}

\definecolor{joecolor}{rgb}{0.0, 0.0, 1.0}

\usepackage{etoolbox}
\makeatletter
\patchcmd{\@maketitle}
  {\ifx\@empty\@dedicatory}
  {\ifx\@empty\@date \else {\vskip3ex \centering\footnotesize\@date\par\vskip1ex}\fi
  \ifx\@empty\@dedicatory}
  {}{}
\patchcmd{\@adminfootnotes}
  {\ifx\@empty\@date\else \@footnotetext{\@setdate}\fi}
  {}{}{}
\makeatother

\begin{document}

\title{Network Models from Petri Nets with Catalysts}

\author{John C.\ Baez} 
\affiliation{Department of Mathematics, University of California, Riverside CA, 92521, USA}
\affiliation{Centre for Quantum Technologies, National University of Singapore, 117543, Singapore}
\email{baez@math.ucr.edu}

\author{John Foley}
\affiliation{Metron, Inc., 1818 Library St., Suite 600, Reston, VA 20190, USA}
\email{foley@metsci.com}

\author{Joe Moeller}
\affiliation{Department of Mathematics, University of California, Riverside CA, 92521, USA}
\email{moeller@math.ucr.edu}


\maketitle

\begin{abstract}
Petri networks and network models are two frameworks for the compositional design of systems of interacting entities. Here we show how to combine them using the concept of a `catalyst': an entity that is neither destroyed nor created by any process it engages in. In a Petri net, a place is a catalyst if its in-degree equals its out-degree for every transition. We show how a Petri net with a chosen set of catalysts gives a network model. This network model maps any list of catalysts from the chosen set to the category whose morphisms are all the processes enabled by this list of catalysts. Applying the Grothendieck construction, we obtain a category fibered over the category whose objects are lists of catalysts. This category has as morphisms all processes enabled by \emph{some} list of catalysts. While this category has a symmetric monoidal structure that describes doing processes in parallel, its fibers also have premonoidal structures that describe doing one process and then another while reusing the catalysts. 
\end{abstract}

\section{Introduction}

Petri nets are a widely studied formalism for describing collections of entities of different types, and how they turn into other entities \cite{GiraultValk, Peterson}. Network models are a formalism for designing and tasking networks of agents   \cite{NetworkModels, NoncommNetMods}. Here we combine the two. This is worthwhile because while both formalisms involve networks, they serve different functions, and are in some sense complementary.

A Petri net can be drawn as a bipartite directed graph with vertices of two kinds: `places', drawn as circles below, and `transitions' drawn as squares:
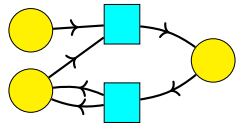
\[
\scalebox{0.8}{
\begin{tikzpicture}
	\begin{pgfonlayer}{nodelayer}
        \node [style=species] (C) at (-1, 0) {$\quad\;$};
        \node [style=species] (B) at (-4, -0.5) {$\quad\;$};
        \node [style=species] (A) at (-4, 0.5) {$\quad\;$};
        \node [style=transition] (tau1) at (-2.5, 0.6) {$\big.\;\;\;\, $};
        \node [style=transition] (tau2) at (-2.5, -0.7) {$\big.\;\;\;\, $};
	\end{pgfonlayer}
	\begin{pgfonlayer}{edgelayer}
        \draw [style=inarrow] (A) to (tau1);
        \draw [style=inarrow] (B) to (tau1);
        \draw [style=inarrow, bend left=15, looseness=1.00] (tau1) to (C);
        \draw [style=inarrow, bend left=15, looseness=1.00] (C) to (tau2);
        \draw [style=inarrow, bend left=15, looseness=1.00] (tau2) to (B); 
        \draw [style=inarrow, bend right =15, looseness=1.00] (tau2) to (B); 
	\end{pgfonlayer}
\end{tikzpicture}
}
\]
In applications to chemistry, places are also called `species'. When we run a Petri net, we start by placing a finite number of `tokens' in each place:
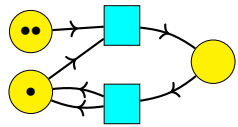
\[
\scalebox{0.8}{
\begin{tikzpicture}
	\begin{pgfonlayer}{nodelayer}
		\node [style=species] (C) at (-1, 0) {$\quad\;$};
		\node [style=species] (B) at (-4, -0.5) {$\;\bullet\;$};
		\node [style=species] (A) at (-4, 0.5) {$\bullet\bullet$};
		\node [style=transition] (tau1) at (-2.5, 0.6) {$\big.\;\;\;\, $};
        \node [style=transition] (tau2) at (-2.5, -0.7) {$\big.\;\;\;\, $};
	\end{pgfonlayer}
	\begin{pgfonlayer}{edgelayer}
		\draw [style=inarrow] (A) to (tau1);
		\draw [style=inarrow] (B) to (tau1);
		\draw [style=inarrow, bend left=15, looseness=1.00] (tau1) to (C);
        \draw [style=inarrow, bend left=15, looseness=1.00] (C) to (tau2);
        \draw [style=inarrow, bend left=15, looseness=1.00] (tau2) to (B); 
        \draw [style=inarrow, bend right =15, looseness=1.00] (tau2) to (B); 
	\end{pgfonlayer}
\end{tikzpicture}
}
\]
This is called a `marking'. Then we repeatedly change the marking using the transitions. For example, the above marking can change to this:
\[
\scalebox{0.8}{
\begin{tikzpicture}
	\begin{pgfonlayer}{nodelayer}
        \node [style=species] (C) at (-1, 0) {$\;\bullet\;$};
        \node [style=species] (B) at (-4, -0.5) {$\quad\;$};
        \node [style=species] (A) at (-4, 0.5) {$\;\bullet\;$};
        \node [style=transition] (tau1) at (-2.5, 0.6) {$\big.\;\;\;\, $};
        \node [style=transition] (tau2) at (-2.5, -0.7) {$\big.\;\;\;\, $};
	\end{pgfonlayer}
	\begin{pgfonlayer}{edgelayer}
		\draw [style=inarrow] (A) to (tau1);
		\draw [style=inarrow] (B) to (tau1);
		\draw [style=inarrow, bend left=15, looseness=1.00] (tau1) to (C);
        \draw [style=inarrow, bend left=15, looseness=1.00] (C) to (tau2);
        \draw [style=inarrow, bend left=15, looseness=1.00] (tau2) to (B); 
        \draw [style=inarrow, bend right =15, looseness=1.00] (tau2) to (B); 
	\end{pgfonlayer}
\end{tikzpicture}
}
\]
and then this:
\[
\scalebox{0.8}{
\begin{tikzpicture}
	\begin{pgfonlayer}{nodelayer}
		\node [style=species] (C) at (-1, 0) {$\quad\;$};
		\node [style=species] (B) at (-4, -0.5) {$\bullet\bullet$};
		\node [style=species] (A) at (-4, 0.5) {$\;\bullet\;$};
		\node [style=transition] (tau1) at (-2.5, 0.6) {$\big.\;\;\;\, $};
        \node [style=transition] (tau2) at (-2.5, -0.7) {$\big.\;\;\;\, $};
	\end{pgfonlayer}
	\begin{pgfonlayer}{edgelayer}
		\draw [style=inarrow] (A) to (tau1);
		\draw [style=inarrow] (B) to (tau1);
		\draw [style=inarrow, bend left=15, looseness=1.00] (tau1) to (C);
        \draw [style=inarrow, bend left=15, looseness=1.00] (C) to (tau2);
        \draw [style=inarrow, bend left=15, looseness=1.00] (tau2) to (B); 
        \draw [style=inarrow, bend right =15, looseness=1.00] (tau2) to (B); 
	\end{pgfonlayer}
\end{tikzpicture}
}
\]
Thus, the places represent different \emph{types} of entity, and the transitions describe ways that one collection of entities of specified types can turn into another such collection. 

Network models serve a different function than Petri nets: they are a general tool for working with networks of many kinds. Mathematically a network model is a lax symmetric monoidal functor $G \maps \S(C) \to \Cat$, where $\S(C)$ is the free strict symmetric monoidal category on a set $C$. Elements of $C$ represent different kinds of `agents'. Unlike in a Petri net, we do not usually consider processes where these agents turn into other agents. Instead, we wish to study everything that can be done with a fixed collection of agents. Any object $x \in \S(C)$ is of the form $c_1 \otimes \cdots \otimes c_n$ for some $c_i \in C$; thus, it describes a collection of agents of various kinds. The functor $G$ maps this object to a category $G(x)$ that describes everything that can be done with this collection of agents. 

In many examples considered so far, $G(x)$ is a category whose morphisms are graphs whose nodes are agents of types $c_1, \dots, c_n$. Composing these morphisms corresponds to `overlaying' graphs. Network models of this sort let us design networks where the nodes are agents and the edges are communication channels or shared commitments. In our first paper the operation of overlaying graphs was always commutative \cite{NetworkModels}. Subsequently we introduced a more general noncommutative overlay operation  \cite{NoncommNetMods}. This lets us design networks where each agent has a limit on how many communication channels or commitments it can handle; the noncommutativity allows us to take a `first come, first served' approach to resolving conflicting commitments.

Here we take a different tack: we instead take $G(x)$ to be a category whose morphisms are \emph{processes that the given collection of agents, $x$, can carry out}. Composition of morphisms corresponds to carrying out first one process and then another.

This idea meshes well with Petri net theory, because any Petri net $P$ determines a symmetric monoidal category $FP$ whose morphisms are processes that can be carried out using this Petri net. More precisely, the objects in $FP$ are markings of $P$, and the morphisms are sequences of ways to change these markings using transitions, e.g.:
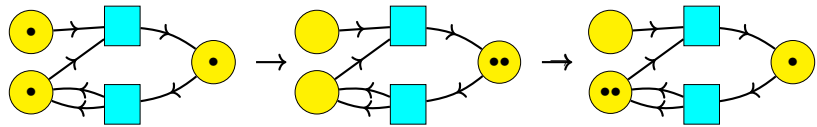
\[
\scalebox{0.8}{
\begin{tikzpicture}
	\begin{pgfonlayer}{nodelayer}
		\node [style=species] (C) at (-1, 0) {$\;\bullet\;$};
		\node [style=species] (B) at (-4, -0.5) {$\;\bullet\;$};
		\node [style=species] (A) at (-4, 0.5) {$\;\bullet\;$};
		\node [style=transition] (tau1) at (-2.5, 0.6) {$\big.\;\;\;\, $};
        \node [style=transition] (tau2) at (-2.5, -0.7) {$\big.\;\;\;\, $};
	\end{pgfonlayer}
	\begin{pgfonlayer}{edgelayer}
		\draw[->, line width=1.00] (-0.3, 0) to (0.2, 0);
		\draw [style=inarrow] (A) to (tau1);
		\draw [style=inarrow] (B) to (tau1);
        \draw [style=inarrow, bend left=15, looseness=1.00] (tau1) to (C);
        \draw [style=inarrow, bend left=15, looseness=1.00] (C) to (tau2);
        \draw [style=inarrow, bend left=15, looseness=1.00] (tau2) to (B); 
        \draw [style=inarrow, bend right =15, looseness=1.00] (tau2) to (B); 
	\end{pgfonlayer}
\end{tikzpicture}
\begin{tikzpicture}
	\begin{pgfonlayer}{nodelayer}
        \node [style=species] (C) at (-1, 0) {$\bullet\bullet$};
        \node [style=species] (B) at (-4, -0.5) {$\quad\;$};
        \node [style=species] (A) at (-4, 0.5) {$\quad\;$};
        \node [style=transition] (tau1) at (-2.5, 0.6) {$\big.\;\;\;\, $};
        \node [style=transition] (tau2) at (-2.5, -0.7) {$\big.\;\;\;\, $};
        \node [style=empty] at (0, 0) {{$\to$}};
	\end{pgfonlayer}
	\begin{pgfonlayer}{edgelayer}
        \draw[->, line width=1.00] (-0.3, 0) to (0.2, 0);
        \draw [style=inarrow] (A) to (tau1);
        \draw [style=inarrow] (B) to (tau1);
        \draw [style=inarrow, bend left=15, looseness=1.00] (tau1) to (C);
        \draw [style=inarrow, bend left=15, looseness=1.00] (C) to (tau2);
        \draw [style=inarrow, bend left=15, looseness=1.00] (tau2) to (B); 
        \draw [style=inarrow, bend right =15, looseness=1.00] (tau2) to (B); 
	\end{pgfonlayer}
\end{tikzpicture}
\begin{tikzpicture}
	\begin{pgfonlayer}{nodelayer}
		\node [style=species] (C) at (-1, 0) {$\;\bullet\;$};
		\node [style=species] (B) at (-4, -0.5) {$\bullet\bullet$};
		\node [style=species] (A) at (-4, 0.5) {$\quad\;$};
		\node [style=transition] (tau1) at (-2.5, 0.6) {$\big.\;\;\;\, $};
        \node [style=transition] (tau2) at (-2.5, -0.7) {$\big.\;\;\;\, $};
	\end{pgfonlayer}
	\begin{pgfonlayer}{edgelayer}
		\draw [style=inarrow] (A) to (tau1);
		\draw [style=inarrow] (B) to (tau1);
        \draw [style=inarrow, bend left=15, looseness=1.00] (tau1) to (C);
        \draw [style=inarrow, bend left=15, looseness=1.00] (C) to (tau2);
        \draw [style=inarrow, bend left=15, looseness=1.00] (tau2) to (B); 
        \draw [style=inarrow, bend right =15, looseness=1.00] (tau2) to (B); 
	\end{pgfonlayer}
\end{tikzpicture}
}
\]

Given a Petri net, then, how do we construct a network model $G \maps \S(C) \to \Cat$, and in particular, what is the set $C$? In a network model the elements of $C$ represent different kinds of agents. In the simplest scenario, these agents persist in time. Thus, it is natural to take $C$ to be some set of `catalysts'. In chemistry, a reaction may require a catalyst to proceed, but it neither increases nor decrease the amount of this catalyst present. For a Petri net, `catalysts' are species that are neither increased nor decreased in number by any transition. For example, species $a$ is a catalyst in the following Petri net, so we outline it in red:
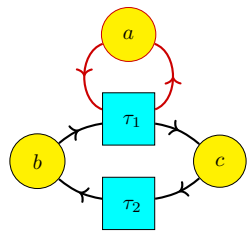
\[
\scalebox{0.8}{
\begin{tikzpicture}
	\begin{pgfonlayer}{nodelayer}
		\node [style=species] (C) at (-1, -0.1) {$\;\;c\;\;$};
		\node [style=species] (B) at (-4, -0.1) {$\;\;b\;\;$};
		\node [style=catalyst] (A) at (-2.5, 2) {$\;\;a\;\;$};
		\node [style=transition] (tau1) at (-2.5, 0.6) {$\;\phantom{\Big{|}}\tau_1\;$};
        \node [style=transition] (tau2) at (-2.5, -0.8){$\;\phantom{\Big{|}}\tau_2\;$};
	\end{pgfonlayer}
	\begin{pgfonlayer}{edgelayer}
		\draw [style=inarrow, bend right=70, looseness=1.00, red] (A) to (tau1);
		\draw [style=inarrow, bend left=15, looseness=1.00] (B) to (tau1);
		\draw [style=inarrow, bend right=70, looseness=1.00, red] (tau1) to (A);
		\draw [style=inarrow, bend left=15, looseness=1.00] (tau1) to (C);
	    \draw [style=inarrow, bend left=15, looseness=1.00] (C) to (tau2);
        \draw [style=inarrow, bend left=15, looseness=1.00] (tau2) to (B); 
	\end{pgfonlayer}
\end{tikzpicture}
}
\]
but neither $b$ nor $c$ is a catalyst. The transition $\tau_1$ requires one token of type $a$ as input to proceed, but it also outputs one token of this type, so the total number of such tokens is unchanged. Similarly, the transition $\tau_2$ requires no tokens of type $a$ as input to proceed, and it also outputs no tokens of this type, so the total number of such tokens is unchanged. 

In Theorem \ref{thm:petri_network_model} we prove that given any Petri net $P$, and any subset $C$ of the catalysts of $P$, there is a network model $G \maps \S(C) \to \Cat$. An object $x \in \S(C)$ says how many tokens of each catalyst are present; $G(x)$ is then the subcategory of $FP$ where the objects are markings that have this specified amount of each catalyst, and morphisms are processes going between these. 

From the functor $G \maps \S(C) \to \Cat$ we can construct a category $\int G$ by `gluing together' all the categories $G(x)$ using the Grothendieck construction. Because $G$ is symmetric monoidal we can use an enhanced version of this construction to make $\int G$ into a symmetric monoidal category \cite{MonGroth}. The tensor product in $\int G$ describes doing processes `in parallel'. The category $\int G$ is similar to $FP$, but it is better suited to applications where agents each have their own `individuality', because $FP$ is actually a \emph{commutative} monoidal category, where permuting agents has no effect at all, while $\int G$ is not so degenerate. In Theorem \ref{thm:grothendieck} we make this precise by more concretely describing $\int G$ as a symmetric monoidal category, and clarifying its relation to $FP$.

There are no morphisms between an object of $G(x)$ and an object of $G(x')$ unless $x \cong x'$, since no transitions can change the amount of catalysts present. The category $FP$ is thus a `disjoint union', or more precisely a coproduct, of subcategories $FP_i$ where $i$, an element of free commutative monoid on $C$, specifies the amount of each catalyst present. The tensor product on $FP$ has the property that tensoring an object in $FP_i$ with one in $FP_j$ gives an object in $FP_{i+j}$, and similarly for morphisms. 

However, in Prop.\ \ref{prop:monoidal} we show that each subcategory $FP_i$ also has its own tensor product, which describes doing one process and then another while reusing catalyst tokens.  This tensor product makes $FP_i$ into a `premonoidal category'---an interesting generalization of a monoidal category which we recall.  Finally, in Theorem \ref{thm:lift} we show that these monoidal structures define a lift of the functor $G \maps \S(C) \to \Cat$ to a functor $\hat{G} \maps \S(C) \to \PreMonCat$, where $\PreMonCat$ is the category of strict premonoidal categories.

\section{Petri Nets}

A Petri net generates a symmetric monoidal category whose objects are tensor products of species and whose morphisms are built from the transitions by repeatedly taking composites and tensor products. There is a long line of work on this topic starting with the papers of Meseguer--Montanari \cite{MM} and Engberg--Winskel \cite{EW}, both dating to roughly 1990. It continues to this day, because the issues involved are surprisingly subtle \cite{DMM, Master,  SassoneStrong, SassoneCategory, SassoneAxiomatization, Congruence}. In particular, there are various kinds of symmetric monoidal categories to choose from. Following our work with Master \cite{OpenPetriNets} we use `commutative' monoidal categories. These are just commutative monoid objects in $\Cat$, so their associator: 
\[  
    \alpha_{a, b, c} \colon (a \otimes b) \otimes c \stackrel{\sim}{\longrightarrow} a \otimes (b \otimes c), 
\]
their left and right unitor:
\[  
    \lambda_a \maps I \otimes a \stackrel{\sim}{\longrightarrow} a , \qquad
    \rho_a \maps a \otimes I \stackrel{\sim}{\longrightarrow} a , 
\]
and even their braiding:
\[   
    \sigma_{a, b} \maps a \otimes b \stackrel{\sim}{\longrightarrow} b \otimes a 
\]
are all identity morphisms. While every symmetric monoidal category is equivalent to one with trivial associator and unitors, this ceases to be true if we also require the braiding to be trivial. However, it seems that  Petri nets most naturally serve to present symmetric monoidal categories of this very strict sort. Thus, we shall describe a functor from the category of Petri nets to the category of commutative monoidal categories, which we call $\CMC$:
\[ 
    F \colon \Petri \to \CMC .
\]

To begin, let $\CMon$ be the category of commutative monoids and monoid homomorphisms.   There is a forgetful functor from $\CMon$ to $\Set$ that sends commutative monoids to their underlying sets and monoid homomorphisms to their underlying functions. It has a left adjoint $\N \maps \Set \to \CMon$ sending any set $X$ to the free commutative monoid on $X$.  An element $a \in \N[X]$ is formal linear combination of elements of $X$:
\[    a = \sum_{x \in X} a_x \, x ,\]
where the coefficients $a_x$ are natural numbers and all but finitely many are zero.  The set $X$ naturally includes in $\N[X]$, and for any function $f \maps X \to Y$, $\N[f] \maps \N[X] \to \N[Y]$ is the unique monoid homomorphism that extends $f$. We often abuse language and use $\N[X]$ to mean the underlying set of the free commutative monoid on $X$. 

\begin{defn} 
    A \define{Petri net} is a pair of functions of the following form:
    \[\begin{tikzcd}
        T
        \arrow[r, shift left = 1, "s"]
        \arrow[r, shift right = 1, "t", swap]
        &
        \N[S].
    \end{tikzcd}\]
    We call $T$ the set of \define{transitions}, $S$ the set of \define{places} or \define{species}, $s$ the \define{source} function, and $t$ the \define{target} function. We call an element of $\N[S]$ a \define{marking} of the Petri net.
\end{defn}
For example, in this Petri net:
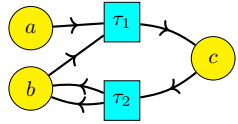
\[
\scalebox{0.8}{
\begin{tikzpicture}
	\begin{pgfonlayer}{nodelayer}
        \node [style=species] (C) at (-1, 0) {$\quad\;$};
        \node [style=species] (B) at (-4, -0.5) {$\quad\;$};
        \node [style=species] (A) at (-4, 0.5) {$\quad\;$};
        \node [style=transition] (tau1) at (-2.5, 0.6) {$\big.\;\;\;\, $};
        \node [style=transition] (tau2) at (-2.5, -0.7) {$\big.\;\;\;\, $};
        \node [style = none] () at (-2.5, 0.6) {$\tau_1$};
        \node [style = none] () at (-2.5, -0.69) {$\tau_2$};
        \node [style = none] () at (-4, 0.5) {$a$};
        \node [style = none] () at (-4, -0.5) {$b$};
        \node [style = none] () at (-1, 0) {$c$};
	\end{pgfonlayer}
	\begin{pgfonlayer}{edgelayer}
        \draw [style=inarrow] (A) to (tau1);
        \draw [style=inarrow] (B) to (tau1);
        \draw [style=inarrow, bend left=15, looseness=1.00] (tau1) to (C);
        \draw [style=inarrow, bend left=15, looseness=1.00] (C) to (tau2);
        \draw [style=inarrow, bend left=15, looseness=1.00] (tau2) to (B); 
        \draw [style=inarrow, bend right =15, looseness=1.00] (tau2) to (B); 
	\end{pgfonlayer}
\end{tikzpicture}
}
\]
we have $S = \{a,b,c\}$, $T = \{\tau_1, \tau_2\}$, and 
\[
\begin{array}{ll} s(\tau_1) = a+b & t(\tau_1) = c \\
  s(\tau_2) = c & t(\tau_2) = 2b.
\end{array}
\]
The term `species' is used in applications of Petri nets to chemistry. Since the concept of `catalyst' also arose in chemistry, we henceforth use the term `species' rather than `places'. 

\begin{defn} 
    A \define{Petri net morphism} from the Petri net $P$ to the Petri net $P'$ is a pair of functions ($f \maps T \to T'$, $g \maps S \to S'$) such that the following diagrams commute:
    \[
    \begin{tikzcd}
        T
        \arrow[r, "s"]
        \arrow[d, "f", swap]
        &
        \N[S]
        \arrow[d, "\N\lbrack g \rbrack"]
        \\
        T'
        \arrow[r, swap, "s'"]
        &
        \N[S']
    \end{tikzcd}
    \quad
    \begin{tikzcd}
        T
        \arrow[r, "t"]
        \arrow[d, "f", swap]
        &
        \N[S]
        \arrow[d, "\N\lbrack g \rbrack"]
        \\
        T'
        \arrow[r, swap, "t'"]
        &
        \N[S']
    \end{tikzcd}
    \]
    Let $\Petri$ denote the category of Petri nets and Petri net morphisms with composition defined by \[(f, g) \circ (f', g') = (f \circ f', g \circ g').\]
\end{defn}

\begin{defn} 
  A \define{commutative monoidal category} is a commutative monoid object in $(\Cat, \times)$. Let $\CMC$ denote the category of commutative monoid objects in $(\Cat,\times)$.
\end{defn}

More concretely, a commutative monoidal category is a strict monoidal category for which $a \otimes b = b \otimes a$ for all pairs of objects and all pairs of morphisms, and the braid isomorphism $a \otimes b \to b \otimes a$ is the identity map.

Every Petri net $P = \left( s, t \maps T \to \N[S] \right)$ gives rise to a commutative monoidal category $FP$ as follows. We take the commutative monoid of objects $\ob(FP)$ to be the free commutative monoid on $S$. We construct the commutative monoid of morphisms $\mor(FP)$ as follows. First we generate morphisms recursively:
\begin{itemize}
    \item for every transition $\tau \in T$ we include a morphism $\tau \maps s(\tau) \to t(\tau)$;
    \item for any object $a$ we include a morphism $1_a \maps a \to a$;
    \item for any morphisms $f \maps a \to b$ and $g \maps a' \to b'$ we include a morphism denoted $f+g \maps a +a' \to b +b'$ to serve as their tensor product;
    \item for any morphisms $f \maps a \to b$ and $g \maps b \to c$ we include a morphism $g\circ f \maps a \to c$ to serve as their composite.
\end{itemize}
Then we quotient by an equivalence relation on morphisms that imposes the laws of a commutative monoidal category, obtaining the commutative monoid $\mor(FP)$.	

Similarly, morphisms between Petri nets give morphisms between their commutative monoidal categories. Given a Petri net morphism 
\[
\begin{tikzcd}
    T
    \arrow[r, shift left = 1]
    \arrow[r, shift right = 1]
    \arrow[d, "f", swap]
    &
    \N[S]
    \arrow[d, "\N\lbrack g\rbrack"]
    \\
    T'
    \arrow[r, shift left = 1]
    \arrow[r, shift right = 1]
    &
    \N[S']
\end{tikzcd}
\]
we define the functor $F(f, g) \maps FP \to FP'$ to be $\N[g]$ on objects, and on morphisms to be the unique map extending $f$ that preserves identities, composition, and the tensor product. This functor is strict symmetric monoidal. 

\begin{prop}
    There is a functor $F \maps \Petri \to \CMC$ defined as above.
\end{prop}

\begin{proof}
This is straightforward; the proof that $F$ is a left adjoint is harder \cite{Master}, but we do not need this here.
\end{proof}

\section{Catalysts}

One thinks of a transition $\tau$ of a Petri net as a process that consumes the source species $s(\tau)$ and produces the target species $t(\tau)$. An example of something that can be represented by a Petri net is a chemical reaction network \cite{BaezBiamonte, BaezPollard}. Indeed, this is why Carl Petri originally invented them. A `catalyst' in a chemical reaction is a species that is necessary for the reaction to occur, or helps lower the activation energy for reaction, but is neither increased nor depleted by the reaction. We use a modest generalization of this notion, defining a \emph{catalyst} in a Petri net to be a species that is neither increased nor depleted by \emph{any} transition in the Petri net. 

Given a Petri net $s, t \maps T \to \N[S]$, recall that for any marking $a \in \N[S]$ we have
\[    a = \sum_{x \in S} a_x x \]
for certain coefficients $a_x \in \N$. Thus, for any transition $\tau$ of a Petri net, $s(\tau)_x$ is the coefficient of the place $x$ in the source of $\tau$, while $t(\tau)_x$ is its coefficient in the target of $\tau$.

\begin{defn}
    A species $x \in S$ in a Petri net $P = (s, t \maps T \to \N[S])$ is called a \define{catalyst} if $s(\tau)_x = t(\tau)_x$ for every transition $\tau \in T$. Let $S_{\mathrm{cat}} \subseteq S$ denote the set of catalysts in $P$.
\end{defn}

\begin{defn} 
    A \define{Petri net with catalysts} is a Petri net $P = (s, t \maps T \to \N[S])$ with a chosen subset $C \subseteq S_{\mathrm{cat}}$. We denote a Petri net $P$ with catalysts $C$ as $(P, C)$.
\end{defn}

Suppose we have a Petri net with catalysts $(P, C)$. Recall that the set of objects of $FP$ is the free commutative monoid $\N[C]$. We have a natural isomorphism 
\[
    \N[S] \cong \N[C] \times \N[S \setminus C]. 
\]
We write
\[
    \pi_C \maps \N[S] \to \N[C] 
\]
for the projection. Given any object $a \in FP$, $\pi_C(a)$ says how many catalysts of each species in $C$ occur in $a$.

\begin{defn} 
    Given a Petri net with catalysts $(P, C)$ and any $i \in \N[C]$, let $FP_i$ be the full subcategory of $FP$ whose objects are objects $a \in FP$ with $\pi_C (a) = i$.
\end{defn}

Morphisms in $FP_i$ describe processes that the Petri net can carry out with a specific fixed amount of every catalyst. Since no transition in $P$ creates or destroys any catalyst, if $f \maps a \to b$ is a morphism in $FP$ then 
\[
    \pi_C(a) = \pi_C(b) .
\]
Thus, $FP$ is the coproduct of all the 
subcategories $FP_i$:
\[
    FP \cong \coprod_{i \in \N[C]} FP_i
\]
as categories. The subcategories $FP_i$ are not generally monoidal subcategories because if $a, b \in FP$ and $a+b$ is their tensor product then 
\[
    \pi_C(a+b) = \pi_C(a) + \pi_C(b) 
\]
so for any $i, j \in \N[C]$ we have
\[
    a \in FP_i, \; b \in FP_j \Rightarrow a + b \in FP_{i + j}
\]
and similarly for morphisms.
Thus, we can think of $FP$ as a commutative monoidal category `graded' by $\N[C]$. But note we are free to reinterpret any process as using a \emph{greater} amount of various catalysts, by tensoring it with identity morphism on this \emph{additional} amount of 
catalysts. That is, given any morphism in $FP_i$, we can always tensor it with the identity on $j$ to get a morphism in $FP_{i+j}$.

Since $\N[C]$ is a commutative monoid we can think of it as a commutative monoidal category with only identity morphisms, and we freely do this in what follows. Network models rely on a similar but less trivial way of constructing a symmetric monoidal category from a set $C$. Namely, for any set $C$ there is a category $\S(C)$ for which:
\begin{itemize}
    \item Objects are formal expressions of the form
    \[
           c_1 \otimes \cdots \otimes c_n 
    \]
    for $n \in \N$ and $c_1, \dots, c_n \in C$. When $n = 0$ we write this expression as $I$.
    \item There exist morphisms
    \[
           f \maps c_1 \otimes \cdots \otimes c_m \to c'_1 \otimes \cdots \otimes c'_n 
    \]
    only if $m = n$, and in that case a morphism is a permutation $\sigma \in S_n$ such that $c'_{\sigma(i)} = c_i$ for all $i=1, \dots , n$.
    \item Composition is the usual composition of permutations. 
\end{itemize}
In short, an object of $\S(C)$ is a list of catalysts, possibly empty, and allowing repetitions. A morphism is a permutation that maps one list to another list. 

As shown in \cite[Prop.\ 17]{NetworkModels}, $\S(C)$ is the free strict symmetric monoidal category on the set $C$. There is thus a strict symmetric monoidal functor
\[
    p \maps \S(C) \to \N[C] 
\]
sending each object $c_1 \otimes \cdots \otimes c_n$ to the object $c_1 + \cdots + c_n$, and sending every morphism to an identity morphism. This can also be seen directly. In what follows, we use this functor $p$ to construct a lax symmetric monoidal functor $G \maps \S(C) \to \Cat$, where $\Cat$ is made symmetric monoidal using its cartesian product.

\begin{prop} 
    Given a Petri net with catalysts $(P, C)$, there exists a unique functor $G \maps \S(C) \to \Cat$ sending each object $x \in \S(C)$ to the category $FP_{p(x)}$ and each morphism in $\S(C)$ to an identity functor.
\end{prop}

\begin{proof} 
    The uniqueness is clear. For existence, note that since $\N[C]$ has only identity morphisms there is a functor $H \maps \N[C] \to \Cat$ sending each object $x \in \N[C]$ to the category $FP_{p(x)}$. If we compose $H$ with the functor $p \maps \S(C) \to \N[C]$ described above we obtain the functor $G$. 
\end{proof}

\begin{thm} 
\label{thm:petri_network_model}
    The functor $G \maps \S(C) \to \Cat$ becomes lax symmetric monoidal with the lax structure map
    \[
        \Phi_{x, y} \maps FP_{p(x)} \times FP_{p(y)} \to FP_{p(x \otimes y)}
    \]
    given by the tensor product in $FP$, and the map
    \[
        \phi \maps 1 \to FP_0 
    \]
    sending the unique object of the terminal category $1 \in \Cat$ to the unit for the tensor product in $FP$, which is the object $0 \in FP_0$.
\end{thm}

\begin{proof} 
    Recall that $G$ is the composite of $p \maps \S(C) \to \N[C]$ and $H \maps \N[C] \to \Cat$. The functor $p$ is strict symmetric monoidal. The functor $p$ is strict symmetric monoidal. One can check that the functor $H$ becomes lax symmetric monoidal if we equip it with the lax structure map
    \[
        FP_i \times FP_j \to FP_{i+j} 
    \]
    given by the tensor product in $FP$, and the map 
    \[
        1 \to FP_0 
    \]
    sending the unique object of $1 \in \Cat$ to the unit for the tensor product in $FP$, namely $0 \in \N[S] = \ob(FP)$. Composing the lax symmetric monoidal functor $H$ and with the strict symmetric monoidal functor $p$, we obtain the lax symmetric monoidal functor $G$ described in the theorem statement. 
\end{proof}

In our previous paper \cite{NetworkModels}, a \define{$C$-colored network model} was defined to be a lax symmetric monoidal functor from $\S(C)$ to $\Cat$. 

\begin{defn}
    We call the $C$-colored network model $G \maps \S(C) \to \Cat$ of Theorem \ref{thm:petri_network_model} the \define{Petri network model} associated to the Petri net with catalysts $(P, C)$.
\end{defn}

\begin{example}
\label{ex:1}
    The following Petri net $P$ has species $S = \{a, b, c, d, e\}$ and transitions $T = \{\tau_1, \tau_2\}$:
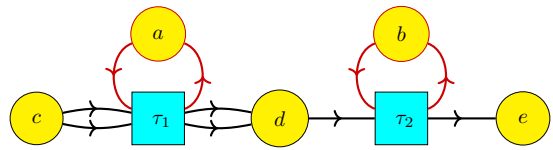
    \[
    \scalebox{0.8}{
    \begin{tikzpicture}
    	\begin{pgfonlayer}{nodelayer}
    	    \node [style=catalyst] (A) at (-2, 2) {$\;\;a\;\;$};
    	    \node [style=catalyst] (B) at (2, 2) {$\;\;b\;\;$};
    		\node [style=species] (C) at (-4, 0.6) {$\;\;c\;\;$};
    		\node [style=species] (D) at (0, 0.6) {$\;\;d\;\;$};
    		\node [style=species] (E) at (4, 0.6) {$\;\;e\;\;$};
    		\node [style=transition] (tau1) at (-2, 0.6) {$\;\phantom{\Big{|}}\tau_1\;$};
            \node [style=transition] (tau2) at (2, 0.6) {$\;\phantom{\Big{|}}\tau_2\;$};
    	\end{pgfonlayer}
    	\begin{pgfonlayer}{edgelayer}
    		\draw [style=inarrow, bend right=70, looseness=1.00, red] (A) to (tau1);
    		\draw [style=inarrow, bend right=70, looseness=1.00, red] (tau1) to (A);
    		\draw [style=inarrow, bend right=70, looseness=1.00, red] (B) to (tau2);
    		\draw [style=inarrow, bend right=70, looseness=1.00, red] (tau2) to (B);
    		\draw [style=inarrow, bend left=12, looseness=1] (C) to (tau1);
    		\draw [style=inarrow, bend right=12, looseness=1] (C) to (tau1);
    	    \draw [style=inarrow, bend left=12, looseness=1] (tau1) to (D);
    	    \draw [style=inarrow, bend right=12, looseness=1] (tau1) to (D);
    	    \draw [style=inarrow] (D) to (tau2);
    	    \draw [style=inarrow] (tau2) to (E);
    	\end{pgfonlayer}
    \end{tikzpicture}
    }
    \]
    Species $a$ and $b$ are catalysts, and the rest are not. We thus can take $C = \{a, b\}$ and obtain a Petri net with catalysts $(P, C)$, which in turn gives a Petri network model $G \maps \S(C) \to \Cat$.  We outline catalyst species in red, and also draw the edges connecting them to transitions in red.
    
    Here is one possible interpretation of this Petri net. Tokens in $\, c\, $ represent people at a base on land, tokens in $\, d\, $ are people at the shore, and tokens in $\, e\, $ are people on a nearby island. Tokens in $\, a\, $ represent jeeps, each of which can carry two people at a time from the base to the shore and then return to the base. Tokens in $\, b\, $ represent boats that carry one person at a time from the shore to the island and then return. 
    
    Let us examine the effect of the functor $G \maps \S(C) \to \Cat$ on various objects of $\S(C)$. The object $a \in \S(C)$ describes a situation where there is one jeep present but no boats. The category $G(a)$ is isomorphic to $FX$, where $X$ is this Petri net:
    \[
    \scalebox{0.8}
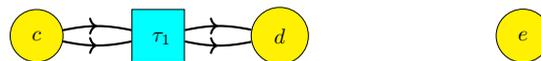
{
    \begin{tikzpicture}
    	\begin{pgfonlayer}{nodelayer}
    		\node [style=species] (C) at (-4, 0.6) {$\;\;c\;\;$};
    		\node [style=species] (D) at (0, 0.6) {$\;\;d\;\;$};
    		\node [style=species] (E) at (4, 0.6) {$\;\;e\;\;$};
    		\node [style=transition] (tau1) at (-2, 0.6) {$\;\phantom{\Big{|}}\tau_1\;$};
    	\end{pgfonlayer}
    	\begin{pgfonlayer}{edgelayer}
    		\draw [style=inarrow, bend left=12, looseness=1] (C) to (tau1);
    		\draw [style=inarrow, bend right=12, looseness=1] (C) to (tau1);
    	    \draw [style=inarrow, bend left=12, looseness=1] (tau1) to (D);
    	    \draw [style=inarrow, bend right=12, looseness=1] (tau1) to (D);
    	\end{pgfonlayer}
    \end{tikzpicture}
    }\]
    That is, people can go from the base to the shore in pairs, but they cannot go to the island. Similarly, the object $b$ describes a situation with one boat present but no jeeps, and the category $G(b)$ is isomorphic to $FY$, where $Y$ is this Petri net:
    \[
    \scalebox{0.8}
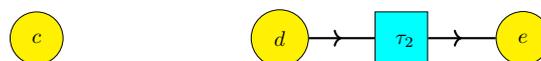
{
    \begin{tikzpicture}
    	\begin{pgfonlayer}{nodelayer}
    		\node [style=species] (C) at (-4, 0.6) {$\;\;c\;\;$};
    		\node [style=species] (D) at (0, 0.6) {$\;\;d\;\;$};
    		\node [style=species] (E) at (4, 0.6) {$\;\;e\;\;$};
            \node [style=transition] (tau2) at (2, 0.6) {$\;\phantom{\Big{|}}\tau_2\;$};
    	\end{pgfonlayer}
    	\begin{pgfonlayer}{edgelayer}
    	    \draw [style=inarrow] (D) to (tau2);
    	    \draw [style=inarrow] (tau2) to (E);
    	\end{pgfonlayer}
    \end{tikzpicture}
    }
    \]
    Now people can only go from the shore to the island, one at a time.
        
    The object $a \otimes b \in \S(C)$ describes a situation with one jeep and one boat. The category $G(a \otimes b)$ is isomorphic to $FZ$ for this Petri net $Z$:
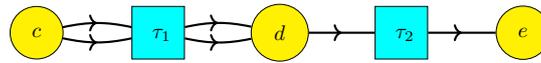
    \[
    \scalebox{0.8}{
    \begin{tikzpicture}
    	\begin{pgfonlayer}{nodelayer}
    		\node [style=species] (C) at (-4, 0.6) {$\;\;c\;\;$};
    		\node [style=species] (D) at (0, 0.6) {$\;\;d\;\;$};
    		\node [style=species] (E) at (4, 0.6) {$\;\;e\;\;$};
    		\node [style=transition] (tau1) at (-2, 0.6) {$\;\phantom{\Big{|}}\tau_1\;$};
            \node [style=transition] (tau2) at (2, 0.6) {$\;\phantom{\Big{|}}\tau_2\;$};
    	\end{pgfonlayer}
    	\begin{pgfonlayer}{edgelayer}
    		\draw [style=inarrow, bend left=12, looseness=1] (C) to (tau1);
    		\draw [style=inarrow, bend right=12, looseness=1] (C) to (tau1);
    	    \draw [style=inarrow, bend left=12, looseness=1] (tau1) to (D);
    	    \draw [style=inarrow, bend right=12, looseness=1] (tau1) to (D);
    	    \draw [style=inarrow] (D) to (tau2);
    	    \draw [style=inarrow] (tau2) to (E);
    	\end{pgfonlayer}
    \end{tikzpicture}
    }
    \]
    Now people can go from the base to the shore in pairs and also go from the shore to the island one at a time.
    
    Surprisingly, an object $x \in \S(C)$ with additional jeeps and/or boats always produces a category $G(x)$ that is isomorphic to one of the three just shown: $G(a), G(b)$ and $G(a \otimes b)$. For example, consider the object $b \otimes b \in \S(C)$, where there are two boats present but no jeeps. There is an isomorphism of categories
    \[  
         - + b \maps G(b)  \to  G(b \otimes b) 
    \]
    defined as follows. Recall that $G(b) = FP_b$ and $G(b \otimes b) = FP_{b+b}$, where $FP_b$ and $FP_{b+b}$ are subcategories of $FP$. The functor
    \[
        - + b \maps FP_b \to FP_{b+b}
    \]
    sends each object $x \in FP_b$ to the object $ x + b$, and sends each morphism $f \maps x \to y$ in $FP_b$ to the morphism $1_b + f \maps b + x \to b + y$. That this defines a functor is clear; the surprising part is that it is an isomorphism. One might have thought that the presence of a second boat would enable one to carry out a given task in more different ways.
        
    Indeed, while this is true in real life, the category $FP$ is \emph{commutative} monoidal, so tokens of the same species have no `individuality': permuting them has no effect. There is thus, for example, no difference between the following two morphisms in $FP_{b+b}$:
    \begin{itemize}
        \item using one boat to transport one person from the base to shore and another boat to transport another person, and
        \item using one boat to transport first one person and then another.
    \end{itemize} 
    
It is useful to draw morphisms in $FP$ as string diagrams, since such diagrams serve as a general notation for morphisms in monoidal categories \cite{JoyalStreet}.   For expository treatments, see \cite{BaezStay,Selinger}.   The rough idea is that objects of a monoidal category are drawn as labelled wires, and a morphism $f \maps x_1 \otimes \cdots \otimes x_m \to y_1 \otimes \cdots \otimes y_n$ is drawn as a box with $m$ wires coming in on top and $n$ wires coming out at the bottom.  Composites of morphisms are drawn by attaching output wires of one morphism to input wires of another, while tensor products of morphisms are drawn by setting pictures side by side.  In symmetric monoidal categories, the braiding is drawn as a crossing of wires.   The rules governing string diagrams let us manipulate them while not changing the morphisms they denote.  In the case of symmetric monoidal categories, these rules are well known \cite{JoyalStreet,Selinger}.   For \emph{commutative} monoidal categories there is one additional rule:
    \[
    \scalebox{0.8}{
    \begin{tikzpicture}
    	\begin{pgfonlayer}{nodelayer}
    		\node [style=empty] (a) at (0, 1) {$x$};
    		\node [style=empty] (b) at (1, 1) {$y$};
    		\node [style=empty] (equals) at (2, 0) {$=$};
            \node [style=empty] (c) at (0, -1) {$y$};
            \node [style=empty] (d) at (1, -1) {$x$};
    	\end{pgfonlayer}
    	\begin{pgfonlayer}{edgelayer}
    		\draw [line width=1.5 pt] (a) to (d);
    		\draw [line width=1.5 pt] (b) to (c);
        	\end{pgfonlayer}
    \end{tikzpicture}
    }
    \scalebox{0.8}{
    \begin{tikzpicture}
    	\begin{pgfonlayer}{nodelayer}
    		\node [style=empty] (a) at (0, 1) {$x$};
    		\node [style=empty] (b) at (1, 1) {$y$};
            \node [style=empty] (c) at (1, -1) {$y$};
            \node [style=empty] (d) at (0, -1) {$x$};
    	\end{pgfonlayer}
    	\begin{pgfonlayer}{edgelayer}
    		\draw [line width=1.5 pt] (a) to (d);
    		\draw [line width=1.5 pt] (b) to (c);
        	\end{pgfonlayer}
    \end{tikzpicture}
    }
    \]
    This says both that $x \otimes y = y \otimes x$ and that the braiding $\sigma_{x,y} \maps x \otimes y \to y \otimes x$ is the identity.
    
    Here is the string diagram notation for the equation we mentioned between two morphisms in $FP$:
    \[
    \scalebox{0.8}
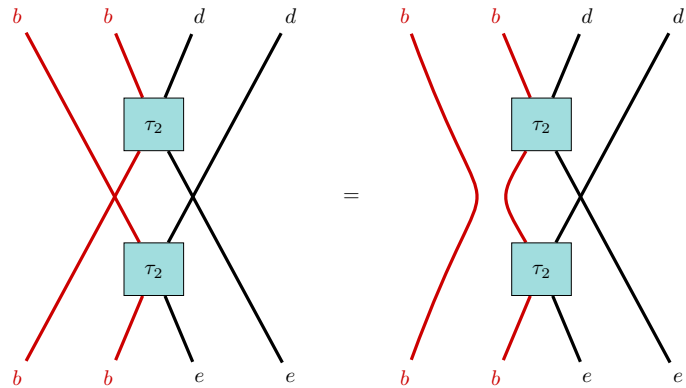
{
    \begin{tikzpicture}
    	\begin{pgfonlayer}{nodelayer}
    		\node [style=empty, red] (b) at (0, 4) {$b$};
    		\node [style=empty, red] (b') at (1.5, 4) {$b$};
    		\node [style=empty] (d) at (3, 4) {$d$};
    		\node [style=empty] (d') at (4.5, 4) {$d$};
    		\node [style=morphism] (tau1) at (2.25, 2.2) {$\;\phantom{\Big|}\tau_2\phantom{\Big|}\;$};
    		\node [style=empty] (equals) at (5.5, 1) {$=$};
            \node [style=morphism] (tau2) at (2.25, -0.2) {$\;\phantom{\Big|}\tau_2\phantom{\Big|\;}$};
            \node [style=empty, red] (b'') at (0, -2) {$b$};
            \node [style=empty, red] (b''') at (1.5, -2) {$b$};
            \node [style=empty] (e) at (3, -2) {$e$};
            \node [style=empty] (e') at (4.5, -2) {$e$};
    	\end{pgfonlayer}
    	\begin{pgfonlayer}{edgelayer}
    		\draw [line width=1.5 pt, red] (b) to (tau2);
    		\draw [line width=1.5 pt, red] (b') to (tau1);
    		\draw [line width=1.5 pt] (d) to (tau1);
    		\draw [line width=1.5 pt] (d') to (tau2);
    		\draw [line width=1.5 pt, red] (tau1) to (b'');
    		\draw [line width=1.5 pt] (tau1) to (e');
    		\draw [line width=1.5 pt, red] (tau2) to (b''');
    		\draw [line width=1.5 pt] (tau2) to (e);
        	\end{pgfonlayer}
    \end{tikzpicture}
    }
    \scalebox{0.8}{
    \begin{tikzpicture}
    	\begin{pgfonlayer}{nodelayer}
    		\node [style=empty, red] (b) at (0, 4) {$b$};
    		\node [style=empty, red] (b') at (1.5, 4) {$b$};
    		\node [style=empty] (d) at (3, 4) {$d$};
    		\node [style=empty] (d') at (4.5, 4) {$d$};
    		\node [style=morphism] (tau1) at (2.25, 2.2) {$\;\phantom{\Big|}\tau_2\phantom{\Big|}\;$};
            \node [style=morphism] (tau2) at (2.25, -0.2) {$\;\phantom{\Big|}\tau_2\phantom{\Big|\;}$};
            \node [style=empty, red] (b'') at (0, -2) {$b$};
            \node [style=empty, red] (b''') at (1.5, -2) {$b$};
            \node [style=empty] (e) at (3, -2) {$e$};
            \node [style=empty] (e') at (4.5, -2) {$e$};
    	\end{pgfonlayer}
    	\begin{pgfonlayer}{edgelayer}
    		\draw [line width=1.5 pt, bend left=20, looseness=2, red] (b) to (b'');
    		\draw [line width=1.5 pt, red] (b') to (tau1);
    		\draw [line width=1.5 pt] (d) to (tau1);
    		\draw [line width=1.5 pt] (d') to (tau2);
    		\draw [line width=1.5 pt, bend right =30, looseness=1.5, red] (tau1) to (tau2);
    		\draw [line width=1.5 pt] (tau1) to (e');
    		\draw [line width=1.5 pt, red] (tau2) to (b''');
    		\draw [line width=1.5 pt] (tau2) to (e);
        	\end{pgfonlayer}
        \end{tikzpicture}
        }
    \]
    We draw the object $b$ (standing for a boat) in red to emphasize that it serves as a catalyst.  At left we are first using one boat to transport one person from the base to shore, and then using another boat to transport another person. At right we are using the same boat to transport first one person and then another, while another boat stands by and does nothing.  These morphisms are equal because they differ only by the presence of the braiding $\sigma_{b,b} \maps b + b \to b + b$ in the left hand side, and this is an identity morphism.
\end{example}

The above example illustrates an important point: in the commutative monoidal category $FP$, permuting catalyst tokens has no effect. Next we construct a symmetric monoidal category $\int G$ in which permuting such tokens has a nontrivial effect.  One reason for wanting this is that in applications, the catalyst tokens may represent agents with their own individuality. For example, when directing a boat to transport a person from base to shore, we need to say \emph{which} boat should do this.  For this we need a symmetric monoidal category that gives the catalyst tokens a nontrivial braiding.

To create this category, we use the symmetric monoidal Grothendieck construction \cite{MonGroth}. Given any symmetric monoidal category $X$ and any lax symmetric monoidal functor $F \maps X \to \Cat$, this construction gives a symmetric monoidal category $\int F$ equipped with a functor (indeed an opfibration) $\int F \to X$. In our previous work \cite{NetworkModels} we used this construction to build an operad from any network model, whose operations are ways to assemble larger networks from smaller ones. Now this construction has a new significance.

Starting from a Petri network model $G \maps \S(C) \to \Cat$, the symmetric monoidal Grothendieck construction gives a symmetric monoidal category $\int G$ in which:
\begin{itemize}
    \item an object is a pair $(x, a)$ where $x \in \S(C)$ and $a \in FP_{p(x)}$.
    \item a morphism from $(x, a)$ to $(x', a')$ is a pair $(\sigma, f)$ where $\sigma \maps x \to x'$ is a morphism in $\S(C)$ and $f \maps a \to a'$ is a morphism in $FP$.
    \item morphisms are composed componentwise.
    \item the tensor product is computed componentwise: in particular, the tensor product of objects $(x, a)$ and $(x', a')$ is $(x \otimes x', a + a')$.
    \item the associators, unitors and braiding are also computed componentwise (and hence are trivial in the second component, since $FP$ is a commutative monoidal category).
\end{itemize}
The functor $\int G \to \S(C)$ simply sends each pair to its first component.

This is simpler than one typically expects from the Grothendieck construction. There are two main reasons: first, $G$ maps every morphism in $\S(C)$ to an identity morphism in $\Cat$, and second, the lax structure map for $G$ is simply the tensor product in $FP$. However, this construction still has an important effect: it makes the process of switching two tokens of the same catalyst species into a nontrivial morphism in $\int G$.
More formally, we have:

\begin{thm}
\label{thm:grothendieck}
    If $G \maps \S(C) \to \Cat$ is the Petri network model associated to the Petri net with catalysts $(P, C)$, then $\int G$ is equivalent, as a symmetric monoidal category, to the full subcategory of $\S(C) \times FP$ whose objects are those of the form $(x, a)$ with $x \in \S(C)$ and $a \in FP_{p(x)}$. 
\end{thm}

\begin{proof} 
    One can read this off from the description of $\int G$ given above.
\end{proof}

The difference between $\int G$ and $FP$ is that the former category keeps track of processes where catalyst tokens are permuted, while the latter category treats them as identity morphisms.  In the terminology of Glabbeek and Plotkin, $\int G$  implements the `individual token philosophy' on catalysts, in which permuting tokens of the same catalyst is regarded as having a nontrivial effect \cite{GlabbeekPlotkin}.  By contrast, $FP$ implements the `collective token philosophy', where all that matters is the \emph{number} of tokens of each catalyst, and permuting them has no effect.  

There is a map from $\int G$ to $FP$ that forgets the individuality of the catalyst tokens. A morphism in $\int G$ is a pair $(\sigma, f)$ where $\sigma \maps x \to x'$ is a morphism in $\S(C)$ and $f \maps a \to a'$ is a morphism in $FP$ with $a \in G(x), a' \in G(x')$. There is a symmetric monoidal functor
\[
    \textstyle{\int} G \to FP 
\]
that discards this extra information, mapping $(\sigma, f)$ to $f$. The symmetric monoidal Grothendieck construction also gives a symmetric monoidal functor 
\[ \textstyle{\int} G \to \S(C)  \]
and this maps $(\sigma, f)$ to $\sigma$. This functor is an opfibration on general grounds \cite{MonGroth}.

\begin{example}
\label{ex:2}
 Let $(P, C)$ be the Petri net with catalysts in Ex.\ \ref{ex:1}, and $G \maps \S(C) \to \Cat$ the resulting Petri network model. In $\int G$ the following two morphisms are not equal:
    \[
    \scalebox{0.8}
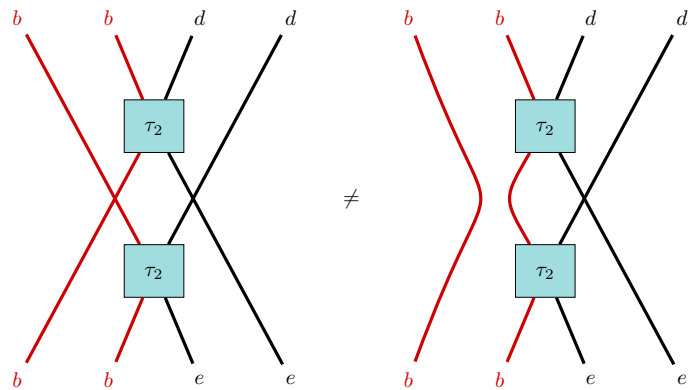
{
    \begin{tikzpicture}
    	\begin{pgfonlayer}{nodelayer}
    		\node [style=empty, red] (b) at (0, 4) {$b$};
    		\node [style=empty, red] (b') at (1.5, 4) {$b$};
    		\node [style=empty] (d) at (3, 4) {$d$};
    		\node [style=empty] (d') at (4.5, 4) {$d$};
    		\node [style=morphism] (tau1) at (2.25, 2.2) {$\;\phantom{\Big|}\tau_2\phantom{\Big|}\;$};
    		\node [style=empty] (equals) at (5.5, 1) {$\ne$};
            \node [style=morphism] (tau2) at (2.25, -0.2) {$\;\phantom{\Big|}\tau_2\phantom{\Big|\;}$};
            \node [style=empty, red] (b'') at (0, -2) {$b$};
            \node [style=empty, red] (b''') at (1.5, -2) {$b$};
            \node [style=empty] (e) at (3, -2) {$e$};
            \node [style=empty] (e') at (4.5, -2) {$e$};
    	\end{pgfonlayer}
    	\begin{pgfonlayer}{edgelayer}
    		\draw [line width=1.5 pt, red] (b) to (tau2);
    		\draw [line width=1.5 pt, red] (b') to (tau1);
    		\draw [line width=1.5 pt] (d) to (tau1);
    		\draw [line width=1.5 pt] (d') to (tau2);
    		\draw [line width=1.5 pt, red] (tau1) to (b'');
    		\draw [line width=1.5 pt] (tau1) to (e');
    		\draw [line width=1.5 pt, red] (tau2) to (b''');
    		\draw [line width=1.5 pt] (tau2) to (e);
    	\end{pgfonlayer}
    \end{tikzpicture}
    }
    \scalebox{0.8}{
    \begin{tikzpicture}
    	\begin{pgfonlayer}{nodelayer}
    		\node [style=empty, red] (b) at (0, 4) {$b$};
    		\node [style=empty, red] (b') at (1.5, 4) {$b$};
    		\node [style=empty] (d) at (3, 4) {$d$};
    		\node [style=empty] (d') at (4.5, 4) {$d$};
    		\node [style=morphism] (tau1) at (2.25, 2.2) {$\;\phantom{\Big|}\tau_2\phantom{\Big|}\;$};
            \node [style=morphism] (tau2) at (2.25, -0.2) {$\;\phantom{\Big|}\tau_2\phantom{\Big|\;}$};
            \node [style=empty, red] (b'') at (0, -2) {$b$};
            \node [style=empty, red] (b''') at (1.5, -2) {$b$};
            \node [style=empty] (e) at (3, -2) {$e$};
            \node [style=empty] (e') at (4.5, -2) {$e$};
    	\end{pgfonlayer}
    	\begin{pgfonlayer}{edgelayer}
    		\draw [line width=1.5 pt, bend left=20, looseness=2, red] (b) to (b'');
    		\draw [line width=1.5 pt, red] (b') to (tau1);
    		\draw [line width=1.5 pt] (d) to (tau1);
    		\draw [line width=1.5 pt] (d') to (tau2);
    		\draw [line width=1.5 pt, bend right =30, looseness=1.5, red] (tau1) to (tau2);
    		\draw [line width=1.5 pt] (tau1) to (e');
    		\draw [line width=1.5 pt, red] (tau2) to (b''');
    		\draw [line width=1.5 pt] (tau2) to (e);
    	\end{pgfonlayer}
    \end{tikzpicture}
    }
    \]
because the braiding of catalyst species in $\int G$ is nontrivial.  This says that in $\int G$ we consider these two processes as different:
\begin{itemize}
        \item using one boat to transport one person from the base to shore and another boat to transport another person, and
        \item using one boat to transport first one person and then another.
    \end{itemize} 

On the other hand, in $\int G$ we have
    \[
    \scalebox{0.8}
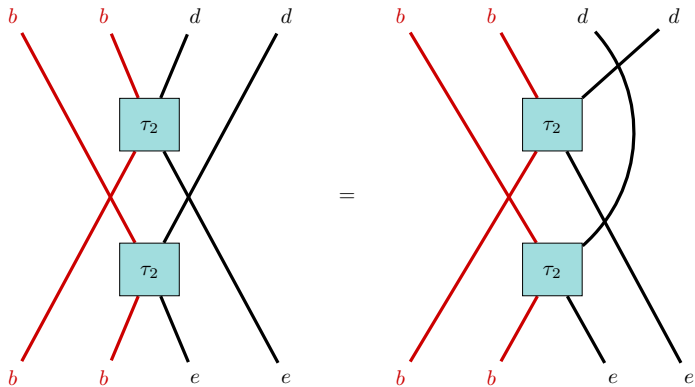
{
    \begin{tikzpicture}
    	\begin{pgfonlayer}{nodelayer}
    		\node [style=empty, red] (b) at (0, 4) {$b$};
    		\node [style=empty, red] (b') at (1.5, 4) {$b$};
    		\node [style=empty] (d) at (3, 4) {$d$};
    		\node [style=empty] (d') at (4.5, 4) {$d$};
    		\node [style=morphism] (tau1) at (2.25, 2.2) {$\;\phantom{\Big|}\tau_2\phantom{\Big|}\;$};
    		\node [style=empty] (equals) at (5.5, 1) {$=$};
            \node [style=morphism] (tau2) at (2.25, -0.2) {$\;\phantom{\Big|}\tau_2\phantom{\Big|\;}$};
            \node [style=empty, red] (b'') at (0, -2) {$b$};
            \node [style=empty, red] (b''') at (1.5, -2) {$b$};
            \node [style=empty] (e) at (3, -2) {$e$};
            \node [style=empty] (e') at (4.5, -2) {$e$};
    	\end{pgfonlayer}
    	\begin{pgfonlayer}{edgelayer}
    		\draw [line width=1.5 pt, red] (b) to (tau2);
    		\draw [line width=1.5 pt, red] (b') to (tau1);
    		\draw [line width=1.5 pt] (d) to (tau1);
    		\draw [line width=1.5 pt] (d') to (tau2);
    		\draw [line width=1.5 pt, red] (tau1) to (b'');
    		\draw [line width=1.5 pt] (tau1) to (e');
    		\draw [line width=1.5 pt, red] (tau2) to (b''');
    		\draw [line width=1.5 pt] (tau2) to (e);
    	\end{pgfonlayer}
    \end{tikzpicture}
    }
      \scalebox{0.8}{
    \begin{tikzpicture}
    	\begin{pgfonlayer}{nodelayer}
    		\node [style=empty, red] (b) at (0, 4) {$b$};
    		\node [style=empty, red] (b') at (1.5, 4) {$b$};
    		\node [style=empty] (d) at (3, 4) {$d$};
    		\node [style=empty] (d') at (4.5, 4) {$d$};
    		\node [style=morphism] (tau1) at (2.5, 2.2) {$\;\phantom{\Big|}\tau_2\phantom{\Big|}\;$};
            \node [style=morphism] (tau2) at (2.5, -0.2) {$\;\phantom{\Big|}\tau_2\phantom{\Big|\;}$};
            \node [style=empty, red] (b'') at (0, -2) {$b$};
            \node [style=empty, red] (b''') at (1.5, -2) {$b$};
            \node [style=empty] (e) at (3.5, -2) {$e$};
            \node [style=empty] (e') at (4.75, -2) {$e$};
    	\end{pgfonlayer}
    	\begin{pgfonlayer}{edgelayer}
    		\draw [line width=1.5 pt, red] (b) to (tau2);
    		\draw [line width=1.5 pt, red] (b') to (tau1);
    		\draw [line width=1.5 pt, bend left=45, looseness=1] (d) to (tau2);
    		\draw [line width=1.5 pt] (d') to (tau1);
    		\draw [line width=1.5 pt, red] (tau1) to (b'');
    		\draw [line width=1.5 pt] (tau1) to (e');
    		\draw [line width=1.5 pt, red] (tau2) to (b''');
    		\draw [line width=1.5 pt] (tau2) to (e);
    	\end{pgfonlayer}
    \end{tikzpicture}
    }
    \]
because these morphisms differ only by two people on the shore switching place before they board the boats, and the braiding of non-catalyst species is the identity.  In short, the $\int G$ construction implements the individual token philosophy only for catalyst tokens; tokens of other species are governed by the collective token philosophy.
\end{example}

\section{Premonoidal Categories}

We have seen that for a Petri net $P$, a choice of catalysts $C$ lets us write the category $FP$ as a coproduct of subcategories $FP_i$, one for each possible amount $i \in \N[C]$ of the catalysts. The subcategory $FP_i$ is only a \emph{monoidal} subcategory when $i = 0$. Indeed, only $FP_0$ contains the monoidal unit of $FP$.  However, we shall see that each subcategory $FP_i$ can be given the structure of a \emph{premonoidal} category, as defined by Power and Robinson \cite{PowerRobinson}.  We motivate our use of this structure by describing two failed attempts to make $FP_i$ into a monoidal category.

Given two morphisms in $FP_i$ we typically cannot carry out these two processes simultaneously, because of the limited availability of catalysts. But we can do first one and then the other.  For example, imagine that two people are trying to walk through a doorway, but the door is only wide enough for one person to walk through. The door is a resource that is not depleted by its use, and thus a catalyst. Both people can use the door, but not at the same time: they must make an arbitrary choice of who goes first. 

We can attempt to define a tensor product on $FP_i$ using this idea.  Fix some amount of catalysts $i \in \N[C]$. Objects of $FP_i$ are of the form $i + a$ with $a \in \N[S - C]$.  On objects we define
\[(i + a) \otimes_i (i + a' ) = i +  a + a'.\]
The unit object for $\otimes_i$ is therefore $i + 0$, or simply $i$. For morphisms
\begin{align*} 
    f &\maps i + a \to i + b \\ 
    f' &\maps i +  a' \to i + b'
\end{align*}
we define 
\[
    f \otimes_i f' = (f + 1_{b'}) \circ (1_{a} + f').
\]

The tensor product
$f \otimes_i f' = (f + 1_{b'}) \circ (1_{a} + f') $
of morphisms in $FP_i$ involves an arbitrary choice: namely, the choice to do $f'$ first. This is perhaps clearer if we draw this morphism as a string diagram in $FP$.
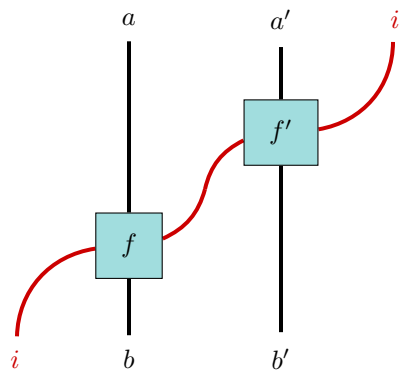
\[
\scalebox{1}{
\begin{tikzpicture}
	\begin{pgfonlayer}{nodelayer}
		\node [style=empty, red] (i) at (0, -3) {$i$};
		\node [style=empty] (a) at (1.5, 1.5) {$a$};
		\node [style=empty] (a') at (3.5, 1.5) {$a'$};
		\node [style=empty] (m) at (2.5, -0.75) {};
		\node [style=morphism] (f) at (1.5, -1.5) {$\;\phantom{\Big|}f\phantom{\Big|}\;$};
        \node [style=morphism] (f') at (3.5, 0) {$\;\phantom{\Big|}f'\phantom{\Big|\;}$};
        \node [style=empty] (b) at (1.5, -3) {$b$};
        \node [style=empty] (b') at (3.5, -3) {$b'$};
        \node [style=empty, red] (i') at (5, 1.5) {$i$};
	\end{pgfonlayer}
	\begin{pgfonlayer}{edgelayer}
		\draw [line width=1.5 pt, bend left = 40, looseness=1, red] (i') to (f');
		\draw [line width=1.5 pt] (a') to (f');
		\draw [line width=1.5 pt] (a) to (f);
		\draw [line width=1.5 pt, bend right = 25, red] (f) to (m.center);
		\draw [line width=1.5 pt, bend left = 25, red] (m.center) to (f');
		\draw [line width=1.5 pt] (f') to (b');
		\draw [line width=1.5 pt, bend right = 40, looseness=1, red] (f) to (i);
		\draw [line width=1.5 pt] (f) to (b);
	\end{pgfonlayer}
\end{tikzpicture}
}
\]
If instead we choose to do $f$ first, we can define a tensor product ${}_i \otimes$ which is the same on objects but given on morphisms by
\[     
    f \, {}_i \!\otimes f' = (1_b + f') \circ (f + 1_{a'})  .
\]
It looks like this:
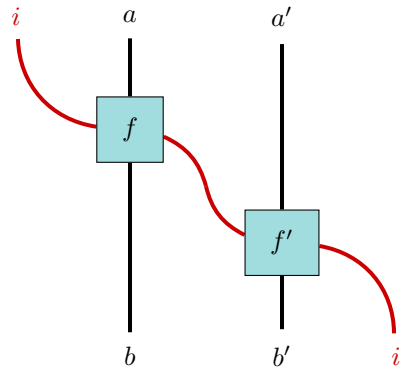
\[
\scalebox{1}{
\begin{tikzpicture}
	\begin{pgfonlayer}{nodelayer}
		\node [style=empty, red] (i) at (0, 1.5) {$i$};
		\node [style=empty] (a) at (1.5, 1.5) {$a$};
		\node [style=empty] (a') at (3.5, 1.5) {$a'$};\node [style=empty] (m) at (2.5, -0.75) {};
		\node [style=morphism] (f) at (1.5, 0) {$\;\phantom{\Big|}f\phantom{\Big|}\;$};
        \node [style=morphism] (f') at (3.5, -1.5) {$\;\phantom{\Big|}f'\phantom{\Big|\;}$};
        \node [style=empty] (b) at (1.5, -3) {$b$};
        \node [style=empty] (b') at (3.5, -3) {$b'$};
        \node [style=empty, red] (i') at (5, -3) {$i$};
	\end{pgfonlayer}
	\begin{pgfonlayer}{edgelayer}
		\draw [line width=1.5 pt, bend right = 40, looseness=1, red] (i) to (f);
		\draw [line width=1.5 pt] (a') to (f');
		\draw [line width=1.5 pt] (a) to (f);
		\draw [line width=1.5 pt, bend left = 25, red] (f) to (m.center);
		\draw [line width=1.5 pt, bend right = 25, red] (m.center) to (f');
		\draw [line width=1.5 pt] (f') to (b');
		\draw [line width=1.5 pt, bend left = 40, looseness=1, red] (f') to (i');
		\draw [line width=1.5 pt] (f) to (b);
	\end{pgfonlayer}
\end{tikzpicture}
}
\]
Unfortunately, neither of these tensor products makes $FP_i$ into a monoidal category!
Each makes the set of objects $\ob(FP_i)$ and the set of morphisms $\mor(FP_i)$ 
into a monoid in such a way that the source and target maps $s,t \maps \mor(FP_i) \to
\ob(FP_i)$, as well as the identity-assigning map $i \maps \ob(FP_i) \to \mor(FP_i)$, are
monoid homomorphisms.  The problem is that neither obeys the interchange law, so
neither of these tensor products defines a functor from $FP_i \times FP_i$ to $FP_i$.   For 
example, 
\[     (1 \otimes_i f')\circ (f \otimes_i 1) \ne (f \otimes_i 1) \circ (1 \otimes_i f') .\]
The other tensor product suffers from the same problem.

What is going on here?  It turns out that $FP_i$ is a `strict premonoidal category'.   While these structures first arose in computer science \cite{PowerRobinson}, they are also mathematically natural, for the following reason.  There are only two symmetric monoidal closed structures on $\Cat$, up to isomorphism \cite{FKL}.  One is the the cartesian product.  The other is the `funny tensor product' \cite{Weber}.   A monoid in $\Cat$ with its cartesian product is a strict monoidal category, but a monoid in $\Cat$ with its funny tensor product is a strict premonoidal category.   
The \define{funny tensor product} $\C \square \D$ of categories $\C$ and $\D$ is defined as the following pushout in $\Cat$:
  \[\begin{tikzcd}
        \C_0 \times \D_0
        \arrow[d, "i \times 1"]   \arrow[r, "1 \times j"]
        &
        \C_0 \times \D
        \arrow[d]
        \\
        \C \times \D_0 
        \arrow[r]
        &
        \C \square \D
  \end{tikzcd}\]
Here $\C_0$ is the subcategory of $\C$ consisting of all the objects and only identity
morphisms, $i \maps \C_0 \to \C$ is the inclusion, and similarly for $j \maps \D_0 \to \D$.
Thus, given morphisms $f \maps x \to y$ in $\C$ and $f' \maps x' \to y'$ in $\C$, the category $C \square D$ in contains a square of the form
\[    \begin{tikzcd}
        x \square x'
        \arrow[d, swap, "f \square 1"]   \arrow[r, "1 \square f' "]
        &
        x \square y'
        \arrow[d,  "f \square 1"]
        \\
        x' \square y
        \arrow[r, swap, "1 \square f' "]
        &
        x' \square y',
  \end{tikzcd}\]
but in general this square does not commute, unlike the corresponding square in
$\C \times \D$.

\begin{defn}
A \define{strict premonoidal category} is a category $\C$ equipped with a functor
$\boxtimes \maps \C \square \C \to \C$ that obeys the associative law and an object $I \in \C$ that serves as a left and right unit for $\boxtimes$. 
\end{defn}

Given two morphisms $f \maps x \to y$, $f' \maps x' \to y'$ in a strict premonoidal category $\C$ we obtain a square
 \[    \begin{tikzcd}
        x \boxtimes x'
        \arrow[d, swap, "f \boxtimes 1"]   \arrow[r,  "1 \boxtimes f' "]
        &
        x \boxtimes y'
        \arrow[d,  "f \boxtimes 1"]
        \\
        x' \boxtimes y
        \arrow[r, swap, "1 \boxtimes f' "]
        &
        x' \boxtimes y',
  \end{tikzcd}\]
but this square may not commute.  There are thus two candidates for a morphism from
$x \boxtimes x'$ to $y \boxtimes y'$.  When these always agree, we can make $\C$
monoidal by setting $f \boxtimes f'$ equal to either (and thus both) of these candidates.   We shall give $FP_i$ a strict premonoidal structure where these two candidates do not agree: one is $f \otimes_i f'$ while the other is $f \, {}_i \! \otimes f'$.  This explains the meaning of these two failed attempts to give $FP_i$ a monoidal structure.

Thanks to the description of $\C \square \C$ as a pushout, to know the tensor product $\boxtimes$ in a strict premonoidal category $\C$ it suffices to know $x \boxtimes y$, $x \boxtimes f$ and $f \boxtimes y$ for all objects $x,y$ and morphisms $f$ of $\C$.  (Here we find it useful to write
$x \boxtimes f$ for $1_x \boxtimes f$ and $f \boxtimes y$ for $f \boxtimes 1_y$.)  In the case at hand, we define
\[    \boxtimes_i \maps FP_i \square FP_i \to FP_i \]
on objects by setting
\[(i + a) \boxtimes_i (i + a' ) = i +  a + a'\]
for all $a,a' \in \N[S-C]$, while for morphisms
\begin{align*} 
    f &\maps i + a \to i + b \\ 
    f' &\maps i +  a' \to i + b'
\end{align*}
we set
\[
	a \boxtimes f' = f' + 1_a, \qquad
     f \boxtimes a' = f + 1_{a'}.
\]

\begin{prop}
\label{prop:monoidal}
    The tensor product $\boxtimes_i$ makes $FP_i$ into a strict premonoidal category.
\end{prop}

\begin{proof}
This can be checked directly, but this is also a special case of a construction in Power and Robinson's paper on premonoidal categories \cite[Ex.\ 3.4]{PowerRobinson}.  They describe a construction, sometimes called `linear state passing' \cite{MogelbergStaton}, that takes any object $i$ in any symmetric monoidal category $C$ and yields a premonoidal category $C_i$ where objects are of the form $i \otimes c$ for $c \in C$ and morphisms are morphisms in $C$ of the form $f \maps i \otimes c \to i \otimes c'$.  We are considering the special case where $C = FP$, and because $FP$ is commutative monoidal the resulting premonoidal category is strict: all the coherence isomorphisms are identities.
\end{proof}

Finally, we show that the tensor products $\boxtimes_i$ on the categories $FP_i$ let us lift our network model $G$ from $\Cat$ to the category of strict premonoidal categories.

\begin{defn}
Let $\PreMonCat$ be the category of strict premonoidal categories and \define{strict premonoidal functors}, meaning functors between strict premonoidal categories that strictly preserve the tensor product.  Let $U \maps \PreMonCat \to \Cat$ denote the forgetful functor which sends a strict premonoidal category to its underlying category.
\end{defn}

\begin{thm}
\label{thm:lift}
    The network model $G \maps \S(C) \to \Cat$ lifts to a functor $\hat G \maps \S(C) \to \PreMonCat$:
    \[\begin{tikzcd}
        &
        \PreMonCat
        \arrow[d, "U"]
        \\
        \S(C)
        \arrow[r, "G", swap]
        \arrow[ur, "\hat G"]
        &
        \Cat
    \end{tikzcd}\]
    where $\hat G(x) = FP_{p(x)}$ with the strict premonoidal structure described in Prop.~\ref{prop:monoidal}.
\end{thm}

\begin{proof}
    Since $G$ sends each morphism in $\S(C)$ to an identity functor, so must $\hat G$. 
\end{proof}

\section{Conclusions}

A couple of mathematical questions arise naturally from this work. First, is there a string diagram calculus for premonoidal categories, like that for monoidal categories but omitting the interchange law?  This is hinted at in the work of Jeffrey \cite{Jeffrey}, but ideally there would be a theorem justifying the use of such string diagrams just as Joyal and Street \cite{JoyalStreet} justified the use of planar progressive string diagrams for monoidal categories.  We could then omit the red lines in our string diagrams and treat the resulting diagrams as describing morphisms in premonoidal categories.

Second, is there a monoidal functor from $(\Cat, \square)$ to $(\Cat, \times)$?   If so, we could turn a strict premonoidal category into a strict monoidal category just by applying this functor. Ideally this would impose the interchange law on the tensor product, forcing all squares of the form
 \[    \begin{tikzcd}
        x \boxtimes x'
        \arrow[d, swap, "f \boxtimes 1"]   \arrow[r,  "1 \boxtimes f' "]
        &
        x \boxtimes y'
        \arrow[d,  "f \boxtimes 1"]
        \\
        x' \boxtimes y
        \arrow[r, swap, "1 \boxtimes f' "]
        &
        x' \boxtimes y'
  \end{tikzcd}\]
to commute.  This could be useful in applications where we do not care which of two processes uses a catalyst first.

\acknowledgments{
This work was supported by the DARPA Complex Adaptive System Composition and Design Environment (CASCADE) project, and we thank Chris Boner, Tony Falcone, Marisa Hughes, Joel Kurucar, Jade Master, Tom Mifflin, John Paschkewitz, Thy Tran and Didier Vergamini for helpful discussions.  We thank Stefano Gogioso for pointing out a problem with a previous version of this paper, and Sam Staton and Noam Zeilberger for helping solve it.  JB also thanks the Centre for Quantum Technologies, where some of this work was done.
}

\bibliographystyle{plainnat}

\end{document}